\documentclass[12pt]{amsart} 
\usepackage[french, english]{babel}
\usepackage[T1]{fontenc}
\usepackage[utf8]{inputenc}
\usepackage{graphicx}
\usepackage{tikz}
\usepackage{hyperref}

\DeclareUnicodeCharacter{2192}{\to}
\DeclareUnicodeCharacter{3C6}{\phi}
\DeclareUnicodeCharacter{221E}{\infty}
\DeclareUnicodeCharacter{2208}{\in}
\DeclareUnicodeCharacter{2265}{\ge}
\DeclareUnicodeCharacter{3A3}{\Sigma}
\usepackage{amsmath,amssymb,latexsym,amsthm,enumerate,mathtools,marginfix} 
\usepackage[abbrev, bibtex-style, nobysame]{amsrefs} 
\usepackage{xparse,etoolbox}
\usepackage[margin=1in]{geometry}

\BibSpec{misc}{
  +{}{\PrintAuthors} {author}
  +{,}{ \textit} {title}
  +{}{ \parenthesize} {date}
  +{,}{ \url} {url}
  +{,}{ } {note}
  +{.}{ } {transition}
}

\usepackage{tikz, tikz-cd} 
\usepackage{microtype} 
\usepackage{geometry}

\renewcommand{\d}{\partial }

\numberwithin{equation}{section}

\newtheorem{Theorem}{Theorem}[section]

\newtheorem{Prop}[Theorem]{Proposition} 
\newtheorem{Corollary}[Theorem]{Corollary} 
\newtheorem{Lemma}[Theorem]{Lemma}

\theoremstyle{definition} 
\newtheorem{Definition}[Theorem]{Definition}
\newtheorem*{overview}{Overview}
\newtheorem*{Acknowledgement}{Acknowledgement}

\theoremstyle{remark} 
\newtheorem{Remark}[Theorem]{Remark}
 
\newtheorem{Example}[Theorem]{Example}

\newcommand{\rr}{\mathbb{R}}
\newcommand{\zz}{\mathbb{Z}}
\newcommand{\nn}{\mathbb{N}}

\newcommand{\cU}{\mathcal{U}}

\DeclareMathOperator\diam{\text{diam}}

\DeclareMathOperator\myc{C}
\DeclareMathOperator\myh{H}

 \DeclareMathOperator\im{im}

\newcommand{\comment}[1]{}

\newcommand{\C}[1][*]{\myc^{#1}}
\newcommand{\Cas}[1][*]{\myc^{#1}_{as}}
\newcommand{\Ccas}[1][*]{\myc^{#1}_{cas}}
\newcommand{\Cx}[1][*]{\myc{\mathrm X}^{#1}}

\newcommand{\Cb}[1][*]{\myc{\mathrm B}^{#1}}
\newcommand{\Cc}[1][*]{\myc^{#1}_{\mathrm c}}

\newcommand{\Cz}[1][*]{\myc^{#1}_{0}}

\newcommand{\cs}[1][*]{\myc_{#1}^{\mathrm s}}

\newcommand{\cf}[1][*]{\myc_{#1}}
\newcommand{\cff}[1][*]{\myc_{#1}^{F}}
\newcommand{\csu}[1][*]{\myc_{#1}^{\cU}}

\renewcommand{\H}[1][*]{\myh^{#1}}
\newcommand{\Hr}[1][*]{\Tilde{\operatorname{H}}\raisebox{0.5pt}{$^{#1}$}}
\newcommand{\Hc}[1][*]{\myh^{#1}_{\mathrm c}}

\newcommand{\Hb}[1][*]{\myh{\mathrm B}^{#1}}
\newcommand{\Hx}[1][*]{\myh{\mathrm X}^{#1}}

\newcommand{\gs}{\sigma }

\newcommand{\gD}{\Delta}

\renewcommand{\ge}{\varepsilon}

\let\oldsubset\subset
\renewcommand{\subset}[1][]{\overset{#1}{\oldsubset}}

\let\oldnotin\notin
\renewcommand{\notin}[1][]{\overset{#1}{\oldnotin}}

\DeclarePairedDelimiter\abs{\lvert}{\rvert}
\newcommand{\Supp}[2][]{\abs{#2}_{#1}}

\DeclarePairedDelimiter\spp{\lVert}{\rVert}

\NewDocumentCommand\ccap{o}{\mathbin{\overset{\mathrm{c}
	\IfNoValueTF{#1}{} {, #1} }
	{\cap}  } }

\NewDocumentCommand\cminus{o}{\mathbin{\oset{\mathrm{c}
	\IfNoValueTF{#1}{} {, #1} }
	{-}  } }	
\NewDocumentCommand\nceq{o}{\mathbin{\overset{\mathrm{c}
	\IfNoValueTF{#1}{} {, #1} }
	{\neq}  } }	
	
\NewDocumentCommand\ceq{o}{\mathbin{\overset{\mathrm{c}
	\IfNoValueTF{#1}{} {, #1} }
	{=}  } }	
	\NewDocumentCommand\csubset{o}{\subset[\mathrm{c}\IfNoValueTF{#1}{} {, #1} ]}
\NewDocumentCommand\he{o}{\mathbin{\overset{ \scriptscriptstyle{#1}}
	{\simeq}  } }	
	\makeatletter
\newcommand{\oset}[3][0ex]{
  \mathrel{\mathop{#3}\limits^{
    \vbox to#1{\kern 0\ex@
    \hbox{$\scriptstyle#2$}\vss}}}}
\makeatother

\begin{document}
\title{On the computation of coarse cohomology}
\author{Arka Banerjee}
\address{Department of Mathematics and Statistics\\
Auburn University\\
Auburn, AL~36849}
\email{azb0263@auburn.edu}
\maketitle
\begin{abstract}
    The purpose of this article is to relate coarse cohomology of  metric spaces with a more computable cohomology.
We introduce a notion of boundedly supported cohomology and 
   prove that coarse cohomology of many spaces are isomorphic to the boundedly supported cohomology.
Boundedly supported cohomology coincides with compactly supported Alexander--Spanier cohomology if the space is proper and contractible.
Our work generalizes an earlier result of Roe which says that the coarse cohomology is isomorphic to the compactly supported Alexander-Spanier cohomology if the space is uniformly contractible.
As an application of our main theorem, we obtain that coarse cohomology of the complement can be computed in terms of Alexander-Spanier cohomology for many spaces.
\end{abstract}
\section{Introduction}

John Roe~\cite{r93} introduced the notion of  coarse cohomology $\Hx(X)$ of a metric space $X$ to study large scale geometry of the space.
This cohomology roughly measures the way in which uniformly large bounded sets fit together.
Coarse cohomology is an invariant of the large scale geometry of the space: if two spaces are coarsely equivalent, then they have isomorphic coarse cohomology.
In general, coarse cohomology is hard to compute. For nice spaces Roe proved the following.
\begin{Theorem}[Roe~\cite{r93}]\label{Roe theorem}
If $X$ is uniformly contractible, then the coarse cohomology $\Hx(X)$ is isomorphic to the compactly supported Alexander--Spanier cohomology $\Hc(X)$.
\end{Theorem}

A space is called \emph{uniformly contractible} if there exists a function $\rho:\rr_+\rightarrow \rr_+$ such that any set of diameter $r$ is contractible inside a set of diameter $\rho(r)$.
Examples of uniformly contractible spaces include universal cover of compact aspherical complexes.
For example, Euclidean space of any dimension is uniformly contractible and proper. By Roe's Theorem it follows that $\Hx(\rr^n)=\Hc(\rr^n)$.

In this article, our goal is to generalize the above theorem to facilitate computation of coarse cohomology for more spaces.
Our main motivation for doing that is to compute coarse cohomology of the complement of a subspace $A\subset X$  denoted by $\Hx(X-A)$.
This notion was introduced and studied in \cite{BB21}.
One way to describe this cohomology is as follows: consider complements of \emph{expanding} neighborhoods $X-N_f(A)$, where $N_{f}(A)= \bigcup_{x\in A} B(x,f(x))$ for proper functions $f:A\rightarrow (0,\infty)$. 
These complements form a directed system of metric spaces under inclusion (in the direction of slower growth of $f$),  and  coarse (co)chain complexes of the complement $X-A$ is defined to be the inverse limits of the usual coarse complexes of these spaces. 

In~\cite{BB21}, it was shown that the coarse cohomology of the complement can be realized as the coarse cohomology of a single space.
Equip $X$ with the following distance function
\[{d_A}(x,y)=\min\{d(x,A)+d(y,A),d(x,y)\}.\]
$d_A$ defines a  pseudometric on $X$. It induces a metric on  the space $X/\Bar{A}$ which is formed by identifying the closure of $A$ to a point.

\begin{Prop}[\cite{BB21}]\label{model for coarse complement}
Let $X$ be a metric space, $A\subset X$ and $X/\Bar{A}$ is the metric space equipped with the metric induced from $d_A$.  Then  $\Hx(X/\Bar{A})$ is isomorphic to $\Hx(X-A)$.
\end{Prop}

So, one way to get the coarse cohomology of the complement of $A$ is to understand the coarse cohomology of  $(X/\Bar{A},d_A)$.
 Unfortunately, Theorem~\ref{Roe theorem} does not apply to the computation of $\Hx(X/\Bar{A})$ because  the space $(X/\Bar{A},d_A)$ is rarely uniformly contractible even if $X$ is.
On the other hand, if $X$ is uniformly contractible, then $(X/\Bar{A},d_A)$ retains those properties at infinity in some sense which we describe next.

We say $X$ is \emph{uniformly contractible at infinity} if there exist two non-decreasing control function $\rho, \mu:\rr_+\rightarrow \rr_+$ and a basepoint $b\in X$ such that any set $B$ of diameter $r$ is contractible inside a set of diameter $\rho(r)$ if $d(b,B)\geq \mu(r)$.
 
Our main theorem in this article is the following.

\begin{Theorem}\label{pmain theorem}
 If $X$ is uniformly contractible at infinity, 
 then its coarse cohomology $\Hx(X)$ is isomorphic to its boundedly supported cohomology $\Hb(X)$. 
\end{Theorem}

We will define $\Hb[*](X)$ later (Definition \ref{bddly supported}) in the paper.
 In particular, when $X$ is proper and contractible, $\Hb(X)$ coincides with the compactly supported Alexander--Spanier cohomology  (Example~\ref{Hb=Hc}).
 Hence, Theorem~\ref{pmain theorem} generalizes Theorem~\ref{Roe theorem}.
In general, $\Hb(X)$ is isomorphic to the reduced Alexander--Spanier cohomology of $X$ at infinity with degree shifted down by one (Proposition \ref{p:bdd=singular}).
As a consequence of Theorem~\ref{pmain theorem}, we prove the following.
\begin{Corollary}\label{application 2}
Suppose $X$ is uniformly contractible at infinity.  
Let $A\subset X$ so that $X\neq N_r(A)$ for any $r$. Then $\Hx(X-A)=\varinjlim \Hr[*-1](X-N_r(A))$ for $*\geq 1$.
\end{Corollary}

The hypothesis of Theorem~\ref{pmain theorem} is much weaker than the hypothesis of Roe's Theorem~\ref{Roe theorem}.
In particular, the property of uniform contractibility at infinity is  more robust compared to uniform contractibility.
For example, deleting a bounded set does not affect the property of being uniformly contractible at infinity.
This is not true for uniformly contractible spaces because such spaces are necessarily contractible and complement of bounded set in a contractible set might not be contractible anymore.
On the other hand, spaces satisfying uniform contractibility at infinity can be very far from being contractible (see Figure~\ref{unif at inf eg}).

\begin{figure}\label{unif at inf eg}
    \centering
    \includegraphics[scale=.7]{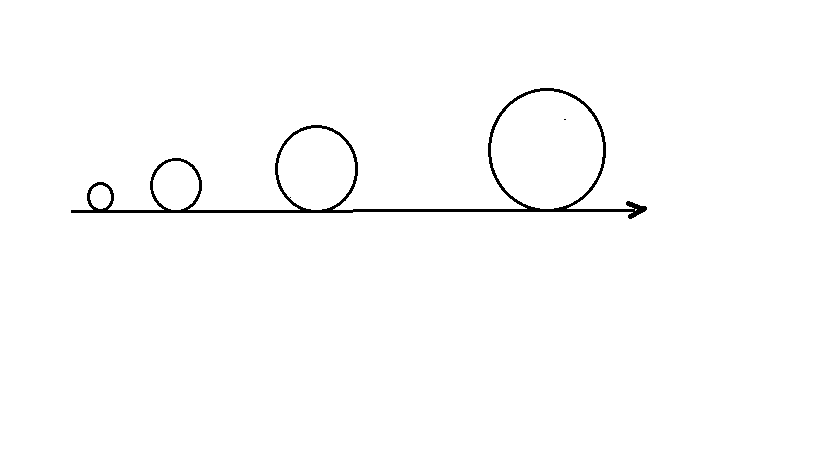}
    \caption{ A  subspace of $\rr^2$ that consists of countable union of circles $\{C_i\}_{i\in\nn}$ and the ray $r:=[0,\infty)\times \{0\}$ such that the $i^{th}$ circle has radius $i$ and distance between two consecutive circle grows to infinity. This is an example of a space that is uniformly contractible at infinity.}
    \label{fig:pack of circles}
\end{figure}

      Our approach to prove Theorem \ref{pmain theorem} is different from the proof of Roe's Theorem~\ref{Roe theorem} in~\cite{r93}.
  The main part of the proof of Theorem~\ref{Roe theorem} in~\cite{r93} involve showing that coarse cohomology is isomorphic to the $\breve{\text{C}}$ech cohomology of certain open covers. In order to show that Roe relates this $\breve{\text{C}}$ech cohomology to certain presheaf cohomology using double complex.
   Our approach to prove Theorem \ref{pmain theorem} follows an idea from \cite{r03} which is more geometric and direct in the sense that we explicitly construct (co)chain homotopy to establish the isomorphism between the concerned cohomology groups.

  A space $X$ is called $\emph{ uniformly acyclic at infinity}$ if there exist a functions $ \rho:[0,\infty) \rightarrow [0,\infty)$ and a basepoint $b\in X$ such that any set $B$ of diameter $r$, the inclusion $B\hookrightarrow N_{\rho(r)}(B)$ induces trivial map between the singular homology.    Another interesting aspect of the Theorem~\ref{pmain theorem} is that it does not hold if we replace uniform contractibility at infinity by uniform acyclicity.
That means any proof of this theorem needed to be able to distinguish between uniform acyclicity and uniform contractibility. This makes the proof more subtle than one might initially expect it to be.

\begin{overview}
In section~\ref{preliminaries}, we set some terminologies for various cochain complexes.
In particular, we define boundedly supported cohomology $\Hb(X)$ and recall coarse cohomology $\Hx(X)$.
In section~\ref{warm up theorem}, we prove a warm-up theorem that says $\Hx(X)$ and $\Hb(X)$ are isomorphic when $X$ is uniformly acyclic at infinity and locally acyclic at infinity.
This proof gives the main ideas that will go into the proof of the Theorem~\ref{pmain theorem}.
Section~\ref{main proof} contains the proof of the Theorem~\ref{pmain theorem}.
Finally in section \ref{coarse and AS}, we prove Corollary \ref{application 2}.
\end{overview}
\begin{Acknowledgement}
Part of this paper was written when the author was a graduate student in University of Wisconsin--Milwaukee. The author would like to thank his advisor Boris Okun for his insights and advice during the project.
In particular, the author is indebted to him for sharing his idea on the proof of Theorem~\ref{easier theorem} from which this whole project originated.
The author would also like to thank Kevin Schreve for helpful conversations.
\end{Acknowledgement}

\section{Preliminaries}\label{preliminaries}

We introduce some notation for various cochain complexes and corresponding cohomology groups associated to $X$.
Let $R$ be a ring.
The  basic complex is the complex  of all cochains with $d:\C[n-1](X;R)\rightarrow \C[n](X;R)$ being the boundary maps as follows:

\begin{align*} 
&\C[n](X;R)=\{\phi: X^{n+1} \to R\}\\
&(d\phi)(x_0,\dots,x_n)=\sum_{i=0}^n(-1)^i\phi(x_0,\dots,\hat{x_i},\dots,x_n)
\end{align*}
It is an acyclic complex.

We will refer to points in $X^{n+1}$ as \emph{$n$-simplices}.
We will refer to a continuous map $f:\Delta^n\rightarrow X$ as a \emph{singular $n$-simplex}.
We will often need to measure distances between simplices of different dimensions.
A convenient way to do this is to stabilize simplices by repeating the last coordinate, as follows.
Let $X^\infty$ denotes the set of all eventually constant sequence in $X$ equipped with the sup metric.
Let $i:X^{n+1} \to X^{\infty} $ denote the map $ (x_{0},\dots,x_{n}) \mapsto (x_{0},\dots,x_{n}, x_{n},x_{n},\dots)$.
For a cochain $\phi$ define its stabilized support $\Supp{\phi} = \{i(\gs) \mid \gs \in X^{n+1} \text{ and } \phi(\gs) \neq 0\} \subset X^{\infty}$.

  Let $\gD=i(X)$ denote the diagonal of $X^{\infty}$. 
Let $||\phi||$ be the intersection of the diagonal $\gD=\{(x,x,\ldots, x)\mid x\in X\}\subset X^{*+1}$ and the closure of the support of the function $\phi:X^{*+1}\rightarrow R$.

\subsection{Alexander--Spanier cohomology}
Let $\Cz(X;R)$ be the complex of locally zero cochains:
\[
\Cz(X;R) = \{\phi \in \C(X;R) \mid \spp{\phi} =\emptyset \}.
\]
 Note that restriction of $d$ gives a well defined map $\C_0(X)\rightarrow \C[*+1]_0(X)$. 
 Consequently, $d$ induces a well defined map $\Cas(X;M)\rightarrow \Cas[*+1](X;R)$, where
 \[\Cas(X;R)=\C(X;R)/\Cz(X;R).\] 
Alexander--Spanier cohomology, denoted by $\H(X;M)$, is the cohomology of the complex $(\Cas(X;M),d)$.

\begin{Example}
    Alexander--Spanier cohomology groups coincide with the singular cohomology groups for locally finite complexes. On the other hand, they may be different for the spaces that are not locally connected. For instance, degree 1 Alexander--Spanier cohomology group  of the Warsaw circle with coefficient in $\zz$ is $\zz$ whereas degree 1 singular cohomology group is trivial.
\end{Example}

\subsection{Compactly supported Alexander--Spanier cohomology} 
Compactly supported cochains are the cochains from the following complex
\[\Cc(X;R) = \{\phi \in \C(X;R) \mid \spp{\phi} \text{ is compact} \}.
\]
Compactly supported Alexander--Spanier cohomology $\Hc(X;R)$ is the cohomology of the following cochain complex 
\[\Ccas(X;R)=\Cc(X;R)/\Cz(X;R).
\]
with the usual boundary operator induced by $d$.

 \subsection{Boundedly supported cohomology}\label{bddly supported}
 Now we define a natural generalization of compactly supported cochains, called boundedly supported cochains:
\[\Cb(X;R):=\{\phi\in \C(X;R)\mid ||\phi|| \text{ is bounded}\}
\] 
 Boundedly supported cohomology  $\Hb(X;R)$ is the cohomology of $\Cb(X;R)$ with the usual boundary operator $d$.

\begin{Example}\label{Hb=Hc} If $X$ is proper and contractible, then
 $\Hb(X;R)=\Hc(X;R)$. 
 Indeed, if $X$ is contractible, then $\Cz(X;R)$ is an acyclic complex. In this case, the sequence
\[
	0 \to \Cz \to \C \to \Ccas \to 0
\]
gives $\Hc(X;R)=\H(\Cc(X;R))$. Moreover, if $X$ is proper, then $\Cc(X;R)=\Cb(X;R)$. 
Hence we get $\Hc(X;R)=\H(\Cb(X;R))=\Hb(X;R)$.
For example, 
\begin{align*}
   \Hb(\rr^n;\zz)=\Hc(\rr^n;\zz)=\begin{cases} \zz & \text{if $*=n$}\\
     0 & \text{otherwise}
     \end{cases}  
 \end{align*}
\end{Example}
 
 In general, the boundedly supported cohomology is the same as the (reduced) Alexander--Spanier cohomology at infinity with degree shifted down by one.
 More precisely we have the following.
 \begin{Prop}\label{p:bdd=singular}
\begin{enumerate}
\item\label{i:bdd} If $X$ is bounded, then 
\[\Hb(X; R)=\begin{cases} R & \text{if  $*=0$}\\   0 & \text{otherwise}
    \end{cases}\]

    \item \label{i:unbdd}If $X$ is unbounded and $b\in X$, then \[\Hb(X; R)=\begin{cases} 0 & \text{if  $*=0$}\\   \varinjlim \Hr[*-1](X-N_r(b); R) & \text{otherwise}
    \end{cases}\]
    
\end{enumerate}
\end{Prop}

\begin{proof}
    \eqref{i:bdd} If $X$ is bounded  then $\Cb(X;R)=\C(X;R)$.
    The cohomology of the latter complex is trivial except in degree 0. 
    Hence, $\Hb(X;R)=0$ for $*\geq 1$.
    $\Hb[0](X;R)=\{\text{constant functions on $X$}\}\cong R$.

\eqref{i:unbdd}  Elements in $\Hb[0](X;R)$ consists of constant functions  defined on $X$ with support contained in a neighborhood of $b$. This means, if $X$ is not bounded, then $\Hb[0](X;R)$ is trivial. Hence in this case $\Hb[0](X;R)=0$.

Consider the following maps between the cochain complexes where $j$ is the inclusion map and $i$ is induced by  canonical restriction maps (followed by quotient maps) $i_r:\C(X;R)\rightarrow \C(X-N_r(b);R) \rightarrow  \Cas(X-N_r(b);R)$.
\[
0\rightarrow \Cb(X;R) \xrightarrow{j} \C(X;R) \xrightarrow{i} \varinjlim \Cas(X- N_r(b);R) \rightarrow 0 
\]
The above is a short exact sequence because of the following 
\begin{align*}
  \ker(i)&=\{\phi\in \C(X;R)\mid \phi \in \Cz( X- N_r(A);R) \text{ for some r}\}\\
&=\{\phi\in \C(X;R)\mid ||\phi|| \text{ is bounded} \}\\&= \im(j)  
\end{align*}
The above short exact sequence of cochain complexes induces a long exact sequence of the corresponding reduced cohomology groups.
The reduced cohomology group of the middle complex is trivial in all degrees.
Hence, the long exact sequence implies that  \[\Hb(X;R) \cong \varinjlim \Hr[*-1](X-N_r(b);R) \text{ for } *\geq 1
\]
\end{proof}
\begin{Example}\label{Hb of pack of circle}
Suppose $X$ is the space appearing in figure \ref{fig:pack of circles}. 
Since $X$ is unbounded, $\Hb[0](X)$ is trivial.
Suppose $b\in X$. For $*\geq 1$, we have the following
\begin{align*}
    \Hb(X,\zz)&=\varinjlim \Hr[*-1](X-N_r(b);\zz)\\
    &=\begin{cases}
         \Pi_{i=0}^\infty\zz/\oplus_{i=0}^\infty\zz \quad &*=2\\
         0 \quad &\text{otherwise}
    \end{cases}
\end{align*}

\end{Example}

We finish this section by briefly recalling the  coarse cohomology.
\subsection{Coarse cohomology}
Coarse cohomology of a space $X$ is the cohomology of the following cochain complex

\[\Cx(X;R):=\{\phi \in \C(X;R) \mid |\phi|\cap N_r(\gD)  \text{ is bounded} \text{ for all }r\}
\]

\begin{Example}
According to Theorem \ref{Roe theorem}, coarse cohomology is isomorphic to compactly supported Alexander--Spanier cohomology for uniformly contractible spaces. In particular, coarse cohomology of the universal cover of any compact aspherical space is isomorphic to its compactly supported Alexander--Spanier cohomology. For example, $\Hx(\rr^n;\zz)=\zz$ for $*=n$ and is trivial otherwise.
\end{Example}

\section{A warm-up theorem}\label{warm up theorem}

A space is called $\emph{locally acyclic at infinity}$ with coefficient in $R$ if complement of some bounded subset in the space  is locally acyclic with coefficients in $R$.
 A space $X$ is $\emph{uniformly}$ $\emph{acyclic at infinity}$ with coefficients in $R$ if there exist two non-decreasing control functions $\mu, \rho:[0,\infty) \rightarrow [0,\infty)$ and a basepoint $b\in X$ such that any set $B$ of diameter $r$ and $d(b,B)\geq \mu(r)$, the inclusion $B\hookrightarrow N_{\rho(r)}(B)$ induces trivial map between the singular homology with coefficients in $R$.
Sometimes, we will refer to such space as $(\mu,\rho)$-uniformly acyclic at infinity.

In this section we prove the following Theorem. 

\begin{Theorem}\label{easier theorem}
	If $X$ is uniformly acyclic at infinity and locally acyclic at infinity with coefficients in $R$, then the inclusion $\Cx(X;R) \hookrightarrow \Cb(X;R)$ induces an isomorphism on cohomology:
	\[
		\Hx(X;R) \cong \Hb(X;R).
	\]
\end{Theorem}

For a uniformly acyclic at infinity space, we can perform a version of the standard ``connect the dots'' construction and the proof of the above theorem relies on that.
Suppose $X$ is $(\mu,\rho)$-uniformly contractible at infinity.
 Every  $1$-simplex $\gs$  of diameter $r$  outside of the ball of radius $\mu(r)$ is fillable, so we can pick a singular $1$-chain $f(\gs)$, such that $\Supp{f(\gs)}\subset N_{\rho(r)}(\gs)$ and $\d f(\gs) = \d \gs$.  
 Proceeding by induction on the dimension,  if a simplex $\gs$ is sufficiently far from the base point, its boundary is already filled by a singular cycle that bounds a singular chain $f(\gs)$ contained in a controlled neighborhood of the simplex. 
 Moreover, if the space is locally acyclic at infinity, we can choose $f(\gs)$ to have small diameter whenever diameter of $\gs$ is small and is outside some bounded set.
 In order to do that we can take the chain $f(\gs)$ to be of diameter $2k(\gs)$ where
 \[
 k(\gs):=\inf\{\diam(|c|)\mid \d c=f(\d\gs)\}.
 \]
 Here we are multiplying by $2$ to ensure  existence of such chain and that whenever $X$ is locally acyclic at $x$, for any open neighborhood $U$ of $x$, there exist another neighborhood $V\subset U$ of $x$, such that $|f(\gs)|\subset U$ for all $\gs\in V^{n+1}$.
 
Note that not every simplex is fillable, and that  the diameter of fillings grow with dimension, as well as the size of the balls that we have to avoid.
To formalize  the notion of sufficiently far, we make the following definition. 
Given an increasing  sequence of increasing control functions $\mu_{n}:[0,\infty)\rightarrow [0,\infty)$, denote
\[
\cff[n](X;R)=\langle \gs^{n} \mid d(\gs^{n},b) \geq  \mu_{n}(\diam \gs^{n}) \rangle \subset \cf[n](X;R)
\]
Since $\mu_{n}$ is increasing, this defines a subcomplex of the chain complex of finitely supported chains.
 Let $V:\cs(X;R) \to \cf(X;R)$ denote the forgetful map, which maps a singular simplex to its vertices.
 For the rest of the section we will suppress the coefficient $R$ from the notation.
 The previous discussion then gives us the following.

\begin{Lemma}\label{l:filling}
    Suppose $X$ is uniformly acyclic at infinity and is locally acyclic at infinity. Then there exist two non-decreasing sequence of control functions $\rho_{n},\mu_n:[0,\infty)\rightarrow [0,\infty)$ and 
 a map $M: \cff(X) \to \cs(X)$ where 
 \[
\cff[n](X)=\langle \gs^{n} \mid d(\gs^{n},A) \geq  \mu_{n}(\diam \gs^{n}) \rangle \subset \cf[n](X)
\]
 such that
\begin{enumerate}
\item \label{i:chain map} $M$ is a chain map.
 \item \label{i:support control} $\Supp{M(\gs^{n})} \subset N_{\rho_{n}(\diam(\gs^n))}( \gs^{n})$.
 \item \label{i:small simplex} There is bounded set $B\subset X$ such that for any $x\notin B$, and a neighborhood $U$ of $x$, there is a neighborhood $W$ of $x$ such that $\Supp{M(\gs^n)}\subset U$ for all $\gs^n\in {W}^{n+1}$.
 \end{enumerate}
 \end{Lemma}

 \begin{Lemma}\label{l:close maps are homotopic}
    Assume that $X\subset Y$. Let $f:\cf(X)\rightarrow \cf(Y)$ be a chain map such that for any $n$-simplex $\gs^n$, $|f(\gs^n)|\subset N_{\rho_n(\diam(\gs^n))}(\gs^n)$ for some non-decreasing sequence of function $\rho_n:[0,\infty)\rightarrow [0,\infty)$. 
     Then $f$ and the inclusion map $i:\cf(X)\rightarrow \cf(Y)$ are homotopic via a chain homotopy $D$ such that $|D(\gs)|\subset N_{\rho_n(\diam(\gs^n))}(\gs^n)$.
 \end{Lemma}

 \begin{proof}
     We can define $D$ by induction on the dimension of $\gs$.
     Define $D_0(x):=(x,f(x))$. Note that $f(x)-x=\d D(x)$ and $D_0(x)\subset N_{\rho_0(\diam \{x\})}(x)$.
     Suppose we have defined $D_m:\cf[m](X)\rightarrow \cf[m+1](X)$ such that $D_m(\gs)\subset N_{\diam(\gs)}(\gs)$ and $\d D_m(\gs)+D_{m-1}\d(\gs)=i(\gs)-f(\gs)$ for any $m\leq n$.
     To define $D_{n+1}(\gs)$ for an $(n+1)$-simplex $\gs$, consider $c=i(\gs)-f(\gs)-D_n\d(\gs)$. 
     By induction hypothesis, $c$ is a cycle and $|c|\subset N_{\rho(\diam(\gs))}(\gs)$.
     Take a vertex $b$ from the chain $c$, and consider the cone operator $T_b:\cf[n+1](X)\rightarrow \cf[n+2](X)$, $(x_0,\ldots,x_{n+1})\mapsto (b,x_0,\ldots,x_{n+1})$.
     Define $D_{n+1}(\gs):=T_b(c)$.
     By construction,  $\d D_{n+1}(\gs)=c=i(\gs)-f(\gs)-D_n\d(\gs)$ and  $|D_{n+1}(\gs)|\subset N_{\rho_n(\diam(\gs))}(\gs)$.
\end{proof}
 
Let $\csu(X)$ be the complex of singular chains supported by $\cU$.
Combining the previous two lemmas we can now prove the following which will be the main ingredient to prove Theorem~\ref{easier theorem}.

\begin{Prop}\label{G and S} 
  Suppose $X$ is uniformly acyclic at infinity and locally acyclic at infinity. Let $\cU$ be an open cover of $X$.
  For each $x\in X$, fix a set $U_x\in \cU$ that contains $x$.
  Then there exist two increasing sequences of control functions $\mu_{n}$ and $\rho_{n}$ and
 a chain map $S: \cff(X) \to \csu(X)$ where
 \[\cff[n](X)=\langle \gs^{n} \mid d(\gs^{n},b)\geq \mu_{n}(\diam \gs^{n}) \rangle \subset \cf[n](X)
 \]
 and a map $G: \cff(X) \to \cf[*+1](X)$ so that
 
\begin{enumerate}
 \item \label{i:pchain} $G$ is a chain homotopy between $VS$ and the inclusion map $i:\cf^F(X)\rightarrow\cf(X)$: 
    \[\d G(\gs)=i(\gs)-VS(\gs)-G\d(\gs).\]
  \item \label{i:pcontrol support} 
  $\Supp{G(\gs^{n})}  \subset N_{\rho_{n}(\diam(\gs^n))}( \gs^{n})$.
 \item \label{i:psmall simplex} There exists a bounded set $B$ such that for any point $x\notin B$ and a neighborhood $U$ of $x$, there is a neighborhood $W$ of $x$ so that $|G(\gs^n)|\subset U^{n+2}$ for any $\gs\in W^{n+1}$.
 \end{enumerate}
\end{Prop} 

\begin{proof}
Let $\mu_*$, $\rho_*$ and $\cff(X)$ be the ones provided by the lemma~\ref{l:filling}.
    We choose a barycentric subdivision map $p:\cs(X)\rightarrow \csu(X)$ and compose it with the map $M$ from lemma~\ref{l:filling} to get the map $S:\cff(X)\rightarrow \csu(X)$.
    Since $|M(\gs^n)|\subset N_{\rho_n(\diam(\gs^n))}(\gs^n)$ and $|p(M(\gs^n))|\subset |M(\gs^n)|$, it follows that $|S(\gs^n)|\subset N_{\rho_n(\diam(\gs))}(\gs^n)$.
   Applying the Lemma~\ref{l:close maps are homotopic} to the map $VS$, we get the homotopy $G$ between $VS$ and $i$ with property \eqref{i:pcontrol support}.
   Finally property \eqref{i:psmall simplex} follows from the property \eqref{i:small simplex} of the map $M$ in Lemma~\ref{l:filling}. 
\end{proof}

We are now ready to prove Theorem~\ref{easier theorem}.
 \begin{proof}[\textbf{Proof of \ref{easier theorem}}]
	By the long exact sequence, we need to show that $\H(\Cb(X)/\Cx(X))=0$.
	That means for $\phi \in \Cb[n](X)$ with $d\phi \in \Cx[n+1](X)$, we need to find $\psi \in \Cb[n-1](X)$ so that $\phi - d\psi \in \Cx[n](X)$.
	Let $U$ be a bounded neighborhood of $ \spp{\phi }$ in $X$ and for each $x \in X - \spp{\phi } $ choose a neighborhood $U_{x}$ such that $U_{x}^{n+1} \cap \Supp{\phi}=\emptyset$.
	Let $\cU$ denote the cover of $X$ that consist of the collection of $U_{x}$ together with $U$.
	
Proposition~\ref{G and S} produces the chain homotopy $G:  \cff(X) \to \cf[*+1](X)$, which we use to define a linear map $D:  \cf(X) \to \cf[*+1](X)$ by setting 
	\[
	D(\gs^{n})=
		\begin{cases}
 			G(\gs^{n}) &\text{ if } \gs^n\in \cff[n](X),\\
			0 &\text{ otherwise.}
		\end{cases}	
	\]

	We  define
	\[
		T=id- \partial D- D \partial,
	\]

	 Dually we have
	\[
		T^{*}= id - D^{*} d - d D^{*} .
	\]

 We claim that $T^{*}\phi\in \Cx[n](X)$.
 By Proposition~\ref{G and S}\eqref{i:pchain}, $T(\gs)=VS(\gs)$ for all $\gs\in \cff[n](X)$.
 If $\diam(\gs^n)\leq k$ and $d(\gs^n,b)>\mu_n(k)$ for some $k\geq 0$, then $\gs^n \in \cff(X)$ and hence $T(\gs^n)=VS(\gs^n)$.
	Moreover, if $\gs^n$ is outside of the $\rho_{n}(k)$-neighborhood of $U^{n+1}$, then $|T(\sigma^n)|$ does not touch $U^{n+1}$ because $|T(\gs^n)|=|VS(\gs^n)|\subset N_{\rho_n(k)}(\gs^n)$. 
 This implies a simplex in $|T(\gs)|$ belongs to $U_x^{n+1}$ for some $x$.
 Therefore
	 $(T^{*}\phi)(\gs^n)=\phi(T (\gs^n))=0$.
  In other words, $\gs^n\notin |T^*\phi|$.
  It follows that, 
  \[|T^*\phi|\cap N_k(\gD)\subset N_{\mu_n(k)}(b)\cup N_{\rho_n(k)}(U^{n+1}).\]
  Since $U$ is bounded by assumption, 
$|T^*\phi|\cap N_k(\gD)$ is bounded.
This proves the claim.

Next we claim that $D^{*}d(\phi')$ is coarse.
  We have $\Supp{G\gs}\subset N_{\rho_{n}(\diam(\gs))}(\gs)$ from Proposition~\ref{G and S}\eqref{i:pcontrol support}.
 Therefore, by construction $\Supp{D\gs}\subset N_{\rho_{n}(\diam(\gs))}(\gs)$.
	So, $D^{*}$ preserves coarseness, and the claim follows since $d(\phi')$ is coarse.
 
Finally we claim that  $D^*(\phi)\in \Cb[n-1](X)$. By Proposition~\ref{G and S}\eqref{i:psmall simplex}, we can choose a bounded set $B$ containing $U$ such that for any point $x\notin B$, and a $U_x\in \cU$ containing $x$, there is  a neighborhood $W$ of $x$ so that $|G(\gs^n)|\subset U_x^{n+2}$ for any $\gs^n\in V^{n+1}$.
Hence for $x\notin  B\cup ||\phi||$, we have $|D_*(\sigma)|\notin |\phi|$ for all $\sigma\in W^{n+1}$.
Therefore, $||D^*(\phi)||\subset ||\phi||\cup B$. The claim follows since $\phi\in \Cb(X)$.

Since $\phi+dD^*(\phi)=T^*(\phi)-D^*d\phi$, our desired cochain is $\psi=-D^{*}(\phi)$.
\end{proof}

Combining Theorem \ref{pmain theorem} and Proposition \ref{p:bdd=singular} we get the following
\begin{Corollary}\label{c:coa=sing}
If $X$ is unbounded,  uniformly acyclic at infinity and locally acyclic at infinity, then for any $b\in X$
\[\Hx(X; R)=\begin{cases} 0 & \text{if  $*=0$}\\   \varinjlim \Hr[*-1](X-N_r(b); R) & \text{otherwise}.
    \end{cases}\]

\end{Corollary}

\begin{Example}
Let $X$ be the space appearing in figure \ref{fig:pack of circles}. 
$X$ is unbounded, uniformly acyclic at infinity and locally acyclic. 
By Theorem~\ref{easier theorem}, we have $\Hx(X)=\Hb(X)$.
It follows from Example~\ref{Hb of pack of circle} that
\begin{align*}
    \Hx(X;\zz)
    &=\begin{cases}
\Pi_{i=0}^\infty\zz/\oplus_{i=0}^\infty\zz \quad &*=2\\
0 \quad &\text{otherwise}.
    \end{cases}
\end{align*}

\end{Example}
\begin{Remark}\label{a curious example}
  Theorem~\ref{easier theorem} does not hold if we drop the locally acyclic at infinity condition.
For instance, consider the space $X=\bigsqcup_{i=1}^\infty S_i$ of disjoint union of countably infinite Warsaw circles.
We can embed $X$ into $\rr^2$ in a way so that $\diam(S_i)=1$ for all $i$ and  $X$ is coarsely equivalent to a ray $[0,\infty)$ with the subspace metric.
 $\Hx[2](X;\zz)=\Hx[2](X;\zz)=\Hx[2]([0,\infty);\zz)=0$.
 On the other hand, each Warsaw circle has nontrivial Alexander--Spanier cohomology in degree 1.
 It follows that, $\varinjlim \Hr[1](X-N_r(b); \zz)\neq 0$.
 So, the conclusion of the above Corollary~\ref{c:coa=sing} fails in this case.
\end{Remark}
\section{Proof of the main theorem}\label{main proof}

In this section we prove our main theorem. We will continue to suppress the coefficient ring $R$ from the notation.
\begin{Theorem}\label{main theorem}
 If $X$ is uniformly contractible at infinity, 
 then the inclusion $\Cx(X) \hookrightarrow \Cb(X)$ induces an isomorphism on cohomology:
	\[
		\Hx(X) \cong \Hb(X).
	\]
\end{Theorem}

The underlying idea of the proof is similar to the proof of Theorem~\ref{easier theorem}.
However, recall from Remark~\ref{a curious example} that Theorem~\ref{main theorem} does not hold if we replace uniform contractibility by uniform acyclicity.
This means any proof of the Theorem~\ref{main theorem} needs to be able to distinguish between uniform
acyclicity and uniform contractibility.

Most of the work will go into  proving Proposition~\ref{main lemma} which is an analog of Proposition~\ref{G and S} when $X$ is uniformly contractible at infinity. 
The main difficulty of getting such analog is the absence of local acyclicity of  $X$ which was present in Proposition~\ref{G and S}.
To get around that, we exploit the Kuratowski's embedding theorem~\cite{Kura} that says any metric space $X$ can be isometrically embedded inside the locally acyclic space $\ell^\infty(X)$ which is the space of  all bounded sequence $\{x_n\}$ in $X$ with the metric $d(\{x_n\},\{y_n\})=\sup_{n\in \nn} \{d(x_n,y_n)\}$.
We then take neighborhood of $X$ inside $\ell^\infty(X)$.
This neighborhood will be locally acyclic but  might not be uniformly acyclic which is another condition we absolutely need.
This tension make the proof of Proposition~\ref{main lemma} more subtle than Proposition~\ref{G and S}.

For convenience, we will use $\ell^\infty$ to mean $\ell^\infty(X)$ for the rest of the paper.

\begin{Prop}\label{main lemma}
 Suppose $X\subset \ell^\infty$ is  uniformly contractible  at infinity. 
   Let $N(X)$ be an open neighborhood of $X$ in $\ell^\infty$. Let $\mathcal{U}$ be an open cover of $N(X)$.
 Then there exist two non-decreasing sequence of control functions $\rho_n, \mu_n:\rr_{\geq 0}\rightarrow \rr_{\geq0}$ and a subcomplex $\cf^F(X)$ of $\cf(X)$ of the following form
\[\cf[n]^F
(X) = \langle \sigma^n \mid  d(\sigma^n, b) > \mu_n(\diam (\sigma^n))\rangle \subset  \cf[n](X),
\]
 a chain map $S:\cf^F(X)\rightarrow \cf^{\mathcal{U}}(N(X))$ and a map $G:\cf^F(X)\rightarrow \cf[*+1](N(X))$  with the following properties.
     \begin{enumerate}

    \item \label{i:mpchain map} $G$ is a chain homotopy between $VS$ and the inclusion map $i:\cf^F(X)\hookrightarrow\cf(N(X))$. 
    \[\d G(\gs)=VS(\gs)-i(\gs)-G\d(\gs)\]
      \item \label{i:mpsupport control} 
         $|G(\sigma^n)|\subset N_{\rho_n(\diam(\sigma^n))}(\sigma^n)$. 
        \item \label{i:mpsmall simplex} For any $k\in \nn$, there exists a bounded set $B\subset X$ such that for any $x\notin B$ and a neighborhood $U$ of $x$ in $\ell^\infty$, there is a neighborhood $W$ of $x$ in $X$ such that $|G(\sigma^n)|\subset U^{n+2}$ for all $\sigma^n\in W^{n+1}$ if $n\leq k$. 
        \label{property 3}
        
    \end{enumerate}
\end{Prop}

\textit{Sketch of the proof:}
As stated before, the main idea of the proof of Proposition~\ref{main lemma} is similar to the Proposition~\ref{G and S}.
To get the map $S$, we first fill simplices in $X$ by singular chains and then use barycentric subdivision on these chains until they fall inside  $\cf^{\mathcal{U}}(X)$. 
Using uniform contractibility of $X$ at infinity, most of this filling process can be done inside $X$  with the necessary control on the support as required by property \eqref{i:mpsupport control}.

But to have property \eqref{i:mpsmall simplex}, we need to fill small simplices by small singular chains.  Unless $X$ is locally acyclic this cannot be achieved by staying inside $X$. 

This is where we use the ambient space $\ell^\infty$.
Since $N(X)$ is an open subset of $\ell^\infty$, we can fill small enough simplices of $X$ in $N(X)$ by taking the convex hull of its vertices.
In summary, we fill `big' simplices  in $X$ and `small' simplices (small ones) in $N(X)$.
  The main difficulty is to choose these fillings in a compatible way so that it gives a chain map $\cf^F(X)\rightarrow \cs(N(X))$: meaning we have to ensure that boundary of the filling is same as filling of the boundary. 
  If $N(X)$ is uniformly contractible, then we can go as before
 by induction on the dimension: to fill in a simplex $\gs$, choose a chain that bounds the filling of $\d \gs$ which is already defined.
  The problem is $N(X)$ might not be even contractible, even if $X$ is uniformly contractible.
   
  To get around this problem, we first construct a chain map $\cf^F(X)\rightarrow \cs(X)$ that sends small simplices to singular chains which can be homotoped to the convex filling by staying inside $N(X)$ (Lemma \ref{thin filling}).
Convex filling of a simplex is the image of that simplex under the following chain map 
\begin{align*}
    c:&\cf(X)\rightarrow \cs(\ell^\infty)\\
    & (x_0,\ldots,x_n)\mapsto  c(\gs):(s_0,\ldots,s_n)\mapsto \sum_{i=0}^ns_i x_i.
\end{align*}
It is in the construction of this chain map $\cff(X)\rightarrow \cs(X)$ where we crucially use the fact that our space is uniformly contractible at infinity, not just uniformly acyclic at infinity.
Figure~\ref{a bad example} illustrates that this construction cannot be performed when $X$ is the Warsaw circle which is an acyclic space.
\begin{figure}
    \centering
    \includegraphics[scale=.45]{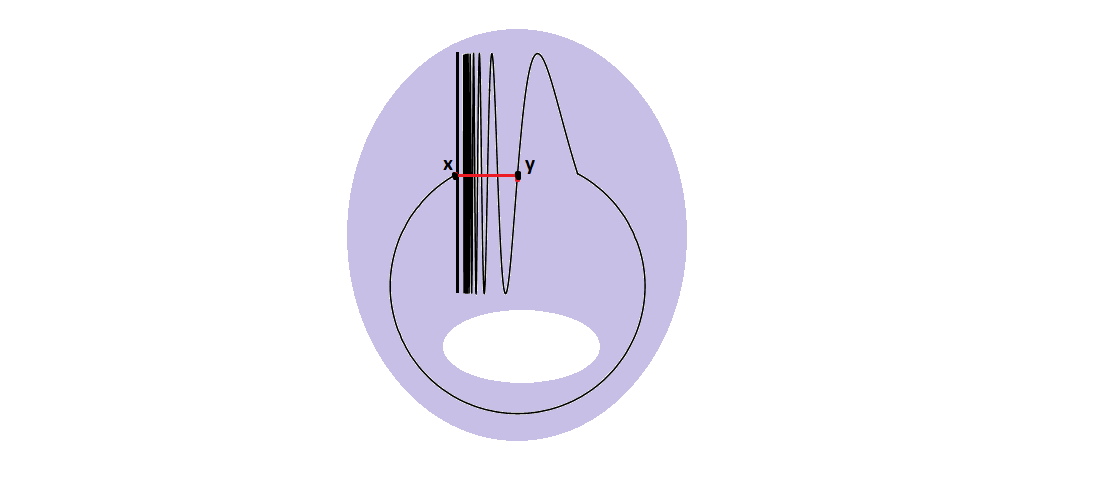}
    \caption{In the above figure, the space $X\subset \rr^2\subset \ell^\infty$ is the Warsaw circle and $N(X)$ is a tubular neighborhood of the grey region in $\ell^\infty$ that deforms retract to the grey region. Red line is the convex filling of the simplex $(x,y)$. Any filling of $(x,y)$ in $X$ has to be around the circular part.
    Hence, there is no homotopy in $N(X)$ between the red convex filling and any filling of $(x,y)$ in $X$. Furthermore, any neighborhood of $x$ contains such a $y$ where this happens.}
    \label{a bad example}
\end{figure}

In the second step, we modify the chain map $\cf^F(X)\rightarrow \cs(X)$ produced in Lemma \ref{thin filling} to get a chain map $\cf^F(X)\rightarrow \cs(N(X))$ that sends small enough simplices to its convex filling. 
The idea here is to glue the filling of small simplices obtained in Lemma~\ref{thin filling} with the associated homotopy of this filling with the convex filling. 
This is done in Lemma~\ref{thin filling 2}.

Finally, to prove Proposition~\ref{main lemma}, we post compose the map $\cf^F(X)\rightarrow \cs(N(X))$  from Lemma \ref{thin filling 2} with a suitable subdivision operator to get the map $S$ into $\cf^\mathcal{U}(X)$ and the homotopy $G$ is obtained by applying Lemma~\ref{l:close maps are homotopic} to $VS:\cf(X)\rightarrow \cf(N(X))$.

\begin{Lemma}\label{thin filling}
Suppose $X\subset \ell^\infty$ is uniformly contractible at infinity. 
Let $N(X)$ be an open neighborhood of $X$ in $\ell^\infty$.
Then there exist two non-decreasing sequence of control functions $\rho_n,\mu_n:\rr_{\geq0}\rightarrow \rr_{\geq0}$, a subcomplex $\cf^F(X)$ of $\cf(X)$ of the following form
\[\cf[n]^F
(X) = \langle \sigma^n \mid  d(\sigma^n; b) > \mu_n(\diam (\sigma^n))\rangle \subset  \cf[n](X),
\]
a chain map $f:\cf^F(X)\rightarrow \cs(X)$ and a map $D:\cf^F(X)\rightarrow \cs[*+1](\ell^\infty)$ with the following properties.
\begin{enumerate}
\item \label{i:mlchain homotopy} $D$ is a chain homotopy between $f$ and the convex filling $c(\gs)$ \[\partial D(\sigma)=f(\sigma)-c(\sigma)-D\partial(\sigma)\]
\item \label{i:mlsupport control} 
$|D(\sigma^n)| \subset N_{\rho_n(\diam(\sigma^n))}(\sigma^n)$ where the tubular neighborhood is taken in $\ell^\infty$.
   
    \item \label{i:mlsmall simplex} For any $k\in \nn$, there exists a bounded set $Z\subset X$ such that for each point $x\in X-Z$, there is a neighborhood $W$ of $x$ in $X$ such that for all $\sigma^n\in W^{n+1}$ with $n\leq k$, $D(\sigma^n)\in \cs[n+1](N(X))$ .
\end{enumerate}
\end{Lemma}

\subsection{Cone construction} We will need to use certain cone operators  to define $f$ and $D$ of the above lemma.
Recall that a standard singular $n$-simplex $\Delta^n$ is the convex hull of the $n+1$ unit vectors in  $\rr^{n+1}$.
A singular $n$-simplex in $X$ is a continuous function $\alpha:\Delta^n\rightarrow X$ and the image of $\alpha$ is called the support of the singular simplex.

Let $H_t$ be a homotopy that contracts some set $B\subset \ell^\infty$ to a point in $\ell^\infty$.
We call such homotopy to be a contracting homotopy. 
To $H_t$, we can associate a cone operator $H:\cs(B)\rightarrow \cs[*+1](\ell^\infty)$.
The construction goes as follows. Let $\alpha$ be a singular $n$-simplex supported in $B\subset \ell^\infty$.
Consider the following map $I\times \Delta^n\rightarrow \ell^\infty$ where $I=[0,1]$.
\[(t,(s_0,\ldots,s_n))\mapsto H_t(\alpha(s_0,\ldots,s_n))
\]
Since $H_1$ is a constant map, the above map induces a map from $(I\times \Delta^n)/(\{1\}\times \Delta^n)$ to $\ell^\infty$. 
$(I\times \Delta^n)/(\{1\}\times \Delta^n)$ can be identified with a singular $(n+1)$-simplex.
Hence the above map gives a singular $(n+1)$-simplex and we define $H(\alpha)$ to be this simplex.
The $i^{th}$ face of $H(\alpha)$ corresponds to the above map restricted to the set of points that have zero in their $i^{th}$ coordinates.
When $\dim(\alpha)\geq 1$, one can check that that $\d(H(\alpha))=\alpha-H(\d \alpha)$.

Now, we will define a different cone operator $\bar{H}$ associated to the homotopy $H_t$.
The inputs of $\bar{H}$ will be generated by convex fillings of simplices whose vertices lie in $B$ (we will later extend the domain of $\bar{H}$).
If $\sigma$ is a  0-simplex, we define $\bar{H}(\sigma):=H(\sigma)$.
If $\sigma=(x_0,\ldots,x_n)$ with $n\geq 1$, 
then we consider the following map $I\times \Delta^n\rightarrow \ell^\infty$.
\[(t,(s_0,\ldots,s_n))\mapsto \sum_{i=0}^n s_i H_t(x_i)
\]
The above map restricted to $\{1\}\times \Delta^n$ is a constant map and hence induces a map from
$(I\times \gD^n)/(\{1\}\times \gD^n)$. 
That means we can realize the above map as a singular $(n+1)$-simplex in $\ell^\infty$, which we define to be $\bar{H}(c(\gs))$ (see figure~\ref{fig:barH}).
The $i^{th}$ face of this simplex corresponds to the above map restricted to the set of points whose $i^{th}$ coordinates are zero.
In particular, the $0^{th}$ face is $c(\gs)$ and hence $\bar{H}(c(\gs))$ can be viewed as a cone on $c(\gs)$.

We can extend the domain of $\bar{H}$ to simplices of the form $\bar{G}(c(\gs))$ where  $\bar{G}$ is the associated cone operator to some contracting homotopy $G_s$ such that the composition $H_t\circ G_s$ is defined on the vertices of an $\gs$ for all $[t,s]\in [0,1]^2$. 

If $\gs$ is a $0$-simplex, then we define $\bar{H}(\bar{G}(c(\gs))):=H(G(\gs))$.
If $\gs=(x_0,\ldots, x_n)$ with $n\geq 1$, then we consider the following map from $I^2\times \gD^n$ to $\ell^\infty$.
\[(t,s,(s_0,\ldots,s_n))\mapsto \sum_{i=0}^ns_i H_t(G_s(x_i))
\]
Since $H_1$ and $G_1$ are constant maps, the above map induce a map from two fold cone on $\gD^n$ and hence gives a singular $(n+2)$-simplex.
We define $\bar{H}(\bar{G}(c(\gs)))$ to be this simplex.

We can iterate the above process.
Suppose $\gs=(x_0,\dots,x_n)$ is an $n$-simplex for some $n\in \nn$.
Let $\{H_{t_k}^k,\dots, H_{t_1}^1\}$ be a set of $k$ contracting homotopies such that the domain of the composition $H^k_{t_k}\circ \dots H^1_{t_1}$ contains the vertices of $\gs$.
When $n=0$, we define \[\bar{H}^k(\dots(\bar{H}^1(\gs)))\dots):=H^k(\dots(H^1(\gs)))\dots)\]
For $n\geq 1$, we consider the following map $I^k\times \gD^n\rightarrow \ell^\infty$
\[(t_k,\dots, t_1,(s_0,\dots,s_n))\mapsto \sum_{i=0}^n s_iH^k_{t_k}(\dots(H^1_{t_1}(x_i)))
\]
Since $H^i_1$ is a constant map for each $i\in \{1,2,\dots,k\}$,
 the above map induces a map from $k$-fold cone on $\gD^n$.
So the above map gives a singular $(n+k)$-simplex and we define this simplex to be  $\bar{H}^k(\dots(\bar{H}^1((c(\gs)))\dots)$. 
Again note that, the above map restricted to the set of points whose first coordinates are zero gives the simplex $\bar{H}^{k-1}(\dots(\bar{H}^1(c(\gs)))\dots)$.

Let us now summarize the above discussion. $\bar{H}$ is a linear map defined on a module generated by singular simplices of the form $\bar{H}^k(\dots(\bar{H}^1(c(\gs)))\dots)$
where the domain of the composition $H_t\circ H^k_{t_{k}}\circ\dots \circ H^1_{t_1}$ contains the vertices of $\gs$ for all $(t,t_k.\dots, t_1)\in [0,1]^{1+k}$.

Suppose $\beta=\bar{H}^k(\dots(\bar{H}^1(c(\gs)))\dots)$ is such a singular simplex in the domain of $\bar{H}$. Then $\bar{H}(\beta)$ is the simplex defined by the following map $I^{k+1}\times \gD^n\rightarrow \ell^\infty$
\[(t,t_k,\dots, t_1,(s_0,\dots,s_n))\mapsto \sum_{i=0}^n s_iH_t(H^{k}_{t_{k}}(\dots(H^1_{t_1}(x_i))\dots))
\]
 In what follows, we will blur the distinction between the above map and the corresponding singular simplex for convenience. We will denote both of them by $\bar{H}(\beta)$.
 The following is immediate from the construction.

\begin{Lemma}\label{boundary of barH}
 If $\dim(\beta)\geq 1$, then $\d \bar{H}(\beta)=\beta-\bar{H}(\d \beta)$.
\end{Lemma}

\begin{figure}
    \centering
    \includegraphics[scale=.65]{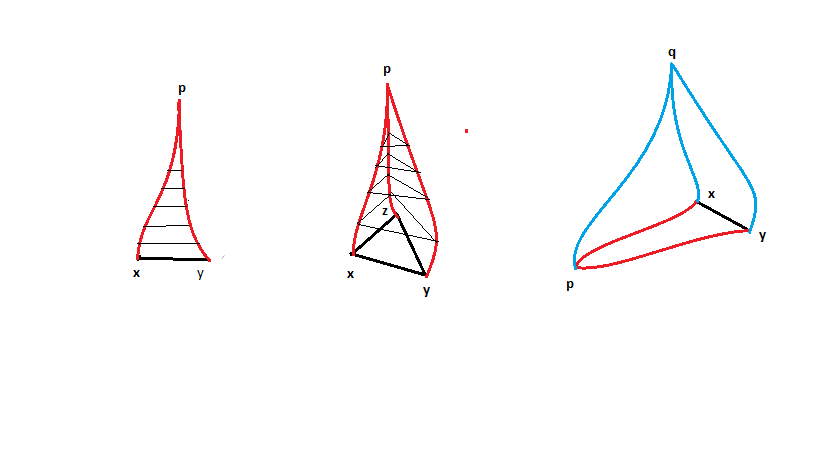}
    \caption{$H_t$ and $G_s$ are two contracting homotopy with $p$ and $q$ being their contracting point respectively. The first two pictures are of $\bar{H}(c(x,y))$ and $\bar{H}(c(x,y,z))$. The third picture is of $\bar{G}(\bar{H}(c(x,y)))$.}
    \label{fig:barH}
\end{figure}

\begin{proof}[\textbf{Proof of \ref{thin filling}}]
Since $X$ is \emph{uniformly contractible at infinity}, there exist two non-decreasing control functions $\rho, \mu:\rr^+\rightarrow \rr^+$ and a basepoint $b\in X$ such that any set $B$ of diameter $r$ is contractible inside a set of diameter $\rho(r)$ if $d(b,B)\geq \mu(r)$.
For each bounded set $B\subset X$ with diameter  $r$ and $d(b,B)\geq \mu(r)$, we fix a homotopy $H^B_t$ that contracts $B$ inside $N_{\rho(r)}(B)$.
Since $X$ is is a metric space it is paracompact. Using paracompactness of $X$, we pick a locally finite open cover $\mathcal{U}=\{U_\alpha\}$ of $X$ such that diameters of all $U_\alpha$ are uniformly bounded.
In particular, for each 1-simplex $(x,y)\in \cup U_\alpha^2$,  the set $S(x,y):=\{\alpha \mid (x,y)\in U^2_\alpha\}$ is finite.
\\

\noindent \textbf{Construction of the complex $\cff(X)$}: We can assume  $\cf[0]^F(X)$ to be $\cf[0](X)$ by letting $\mu_0=0$. For each 1-simplex $\sigma=(x,y)$ in $\cup U_\alpha^2$, we pick an $\alpha$ so that $\gs\in U_\alpha^2$
and we let $B(\sigma):=U_\alpha$.
For other 1-simplices, we let $B(\sigma)$ to be $\{x,y\}$.
Proceeding inductively on dimension, for an $n$-simplex $\sigma$, suppose $B(\tau)$ is already defined for any $\tau\in |\d\gs|$.
Let $r$ be the diameter of $\bigcup_{\tau\in|\d\gs|}B(\tau)$.
We define
\[
B(\sigma)= N_{\rho(r)}(\cup_{\tau\in|\d\gs|}B(\tau)).
\]
We observe that for any $S$, there exist $R$ such that $\diam(B(\sigma))\leq S$ whenever $\diam(\sigma)\leq R$. 
Therefore by uniform contractibility at infinity, we can choose an increasing sequence of functions $\mu_n:[0,\infty)\rightarrow[0,\infty)$ indexed by natural numbers so that $\mu_n\geq R$ for $n\geq 1$ and that $B(\sigma)$ is contractible inside $N_{\rho(\diam(B(\sigma))}(B(\sigma))$ if $d(\sigma^n,b)\geq \mu_n(\diam(\sigma^n))$.
For $n\geq 1$, we let 
\[\cf[n]^F
(X) = \langle \sigma^n \mid  d(\sigma^n; b) > \mu_n(\diam (\sigma^n))\rangle. 
\]
\\

\noindent \textbf{Construction of the chain map $f$}:  We let $f$ to be the identity map on $\cf[0]^F(X)=\cf[0](X)$.
For every $\gs \in \cf^F(X)$, we have an associated homotopy $H_t^{B(\sigma)}$ that contracts $B(\sigma)$ inside $N_{\rho(\diam(B(\sigma))}(B(\sigma))$.
We define $f$ inductively to be the following (see figure~\ref{fig:thinfilling}):
\[
f(\sigma):=\bar{H}^{B(\sigma)}(f(\partial \sigma)).
\]
In order for the above definition to make sense, $f(\d \gs)$ has to be in the domain of $\bar{H}^{B(\gs)}$. 
We can show that by induction on dimension of $\gs$. 
By induction hypothesis, suppose $f(\d \gs)$ is well defined and hence $|f(\d \gs)|$ consists of singular simplices of the form $\bar{H}^{B(\tau)}(\beta)$ where $\tau$ is a codimension one subsimplex of $\gs$.
Note that for any subsimplex $\tau$ of $\gs$, $H_t^{B(\tau)}(B(\tau))\subset N_{\rho(\diam(B(\tau))}B(\tau)\subset B(\gs)$. 
That implies any simplex of the form $\bar{H}^{B(\tau)}(\beta)$ is in the domain of $\bar{H}^{B(\gs)}$ and hence $f$ is well defined.\\

Now we show that $f$ is a chain map by induction.
Notice that  $f|_{\cf[0]^F(X)}$ is the inclusion map $\cf[0]^F(X)\hookrightarrow\cf[0](X)$. 
For a 1-simplex $(x,y)$, we have 
\begin{align*}\partial f((x,y)) &= \partial (\bar{H}^{B((x,y))}(\partial (x,y)))\\
&= \partial (H^{B((x,y))}(y-x))\\
&=\partial (x,y)=f(\partial ((x,y)))
\end{align*}

Inducting  on the dimension of $\sigma$, let us assume  that $f(\partial(\sigma))=\partial(f(\sigma))$ if $1 \leq\dim(\sigma)<n$.
Recall from Lemma~\ref{boundary of barH} that $\partial(\bar{H}^{B(\sigma)}(\alpha))=\alpha- \bar{H}^{B(\sigma)}(\partial \alpha)$ if $\dim(\alpha)\geq 1$.
If $\dim(\gs)\geq 2$, then we have the following.
\begin{align*}\partial f(\sigma)&=\partial \bar{H}^{B(\sigma)}(f(\partial(\sigma)))\\
&=f(\partial(\sigma))-\bar{H}^{B(\sigma)}(\d f(\partial(\sigma))) \quad [\text{by Lemma \ref{boundary of barH}}]\\
&=f(\partial(\sigma)) \end{align*} 
\\
\noindent\textbf{Construction of the chain homotopy $D$}:  
We first let $D_0$ to be trivial on $\cf[0]^F(X)$.
We then define $D_*(\sigma)$ inductively  to be the following.
\[D_*(\sigma):=-\bar{H}^{B(\sigma)}(D_*(\partial \sigma)) - \bar{H}^{B(\sigma)}(c(\sigma))
\]
In order for the above definition to make sense, we need to check that $D_*(\d\gs)$ does belong to the domain of the $\bar{H}^{B(\gs)}$.
The proof of this is same as the  well definedness of $f$, so we will skip it.
(See \ref{fig:thinfilling} for an  illustration of  $f$ and $D$ applied to 1 and 2-dimensional simplices.)
\begin{figure}
    \centering
    \includegraphics[scale=.8]{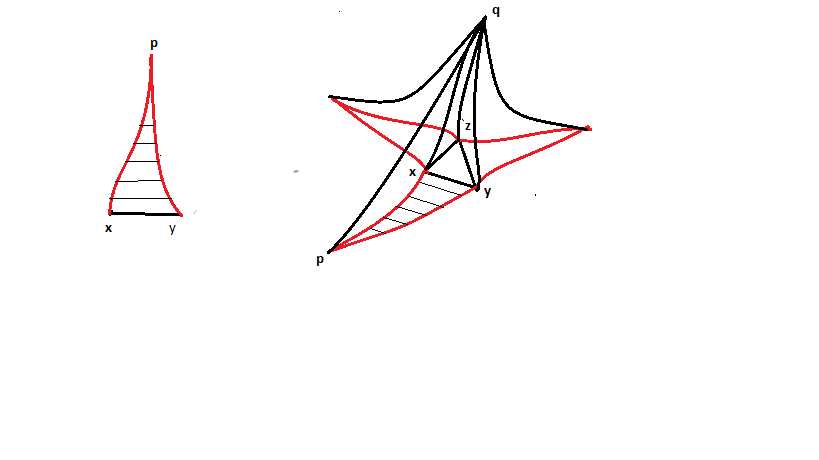}
    \caption{In the left picture, the straight line between $x$ and $y$ is $c((x,y))$ and the red  singular 1-simplices give the filling $f((x,y))=\bar{H}^{B((x,y))}(\partial(x,y))$ where $H^{B((x,y)}_t$ contracts the set $B((x,y))$ to the point $p$. The striped region is $D((x,y))=\bar{H}^{B(x,y)}(c((x,y)))$.
    The picture on the right is the support of $\bar{H}^{B((x,y,z))}(D(\partial((x,y,z)))$ which is made of four singular 3-simplices.
    The one in the center is $\bar{H}^{B(x,y,z)}((x,y,z))$. Other three belong to the support of $\bar{H}^{B((x,y,z))}(D(\partial((x,y,z))))$.
    }
    \label{fig:thinfilling}
    
\end{figure}\\

 \textbf{Proof of \eqref{i:mlchain homotopy}}: This can be proved by induction on the dimension of $\sigma$.
Since $D_0$ is trivial, the claim is true for the base case $n=1$.
For a $1$-simplex $(x,y)$, we have \begin{align*}\partial D((x,y))=-\partial \bar{H}^{B((x,y))}(c((x,y)))&=-c((x,y))+\bar{H}^{B((x,y))}(\partial(c(x,y)))\\&=-c((x,y))+\bar{H}^{B((x,y))}(\partial(x,y))\\&=-c((x,y))+f((x,y))\end{align*}

Assume that $\partial D_*(\sigma)=f(\sigma) - c(\sigma)- D_*(\partial \sigma)$ if $\dim(\sigma)<n$ where $n\geq 2$.
Suppose $\gs$ is an $n$-simplex with $n\geq 2$. Then we have 
\begin{align*}
    &\partial D_*(\sigma)\\
    &=-\partial \bar{H}^{B(\sigma)}(D_*(\partial (\sigma)))-\partial \bar{H}^{B(\sigma)}(c(\sigma))\\
    &=\bar{H}^{B(\sigma)}(\partial(D_*(\partial \sigma)))-D_*(\partial(\sigma))-\partial \bar{H}^{B(\sigma)}(c(\sigma)) && [\text{by Lemma \ref{boundary of barH}}]\\
    &= \bar{H}^{B(\sigma)}(f(\partial(\sigma)))-\bar{H}^{B(\sigma)}(c(\partial \sigma)))-\bar{H}^{B(\sigma)}(D_*(\partial^2(\sigma)))\\
    &\hspace{3cm}-D_*\partial(\sigma)+\bar{H}^{B(\sigma)}(\partial(c(\sigma)))- c(\sigma)
    &&[\text{by induction hypothesis}]\\
    &=\bar{H}^{B(\sigma)}(f(\partial(\sigma)))- c(\sigma)-D_*(\partial(\sigma)) \\
    &=f(\sigma)-c(\sigma)-D_*\partial(\sigma)
\end{align*}
\newline

\noindent \textbf{Proof of \eqref{i:mlsupport control}}: Note that vertices of simplices of $|f(\gs)|$ and $|D(\gs)|$ live in $N_{\rho(\diam(B(\gs))}(B(\gs))$.
As observed earlier, $\diam(B(\gs))$ depends only on $\dim(\gs)$ and $\diam(\gs)$.
That means, we can find a non-decreasing sequence of function $\rho_n:[0,\infty)\rightarrow [0,\infty)$ such that $|f(\gs^n)|\leq \rho_n(\diam(\gs^n))$ and $|D(\gs^n)|\leq \rho_n(\diam(\gs^n))$.
 The claim follows.\\

\noindent \noindent \textbf{Proof of \eqref{i:mlsmall simplex}}:
To prove  \eqref{i:mlsmall simplex}, we will take a closer look at the support of $D(\gs)$.
Consider the following set
\[\mathcal{P}(\gs):=\{(\gs^n,\gs^{n-1},\dots, \gs^j)\mid \gs^j\subset \gs^{j+1}\subset\cdots \gs^{n-1}\subset \gs^n=\gs, 1\leq j\leq n\}
\]

For each  $s=(\gs^n,\dots,\gs^j)\in \mathcal{P}(\gs)$ and $\gs^j=(x_0,\dots,x_j)$, consider the following simplex
\[D_s(\gs):=\bar{H}^{B(\sigma^n)}(\bar{H}^{B(\sigma^{n-1})}(\dots (\bar{H}^{B(\sigma^j)}(c(\gs^j)))\dots)).
\]
We can check that 
\begin{align*}\label{support of D}
    |D(\gs)|=\cup_{s\in \mathcal{P}(\gs)} |D_s(\gs)|
\end{align*}

If $s=(\gs^n,\dots,\gs^j)\in \mathcal{P}(\gs)$ and $\gs^j=(x_0,\dots,x_j)$, then $D_s(\gs)$ is supported in the following set
\[\{\sum_{i=0}^{j} s_i \cdot H_{t_n}^{B(\sigma^n)}(H_{t_{n-1}}^{B(\sigma^{n-1})}(.. (H_{t_j}^{B(\sigma^j)}(x_i)))..) \} \mid  \sum_{i=0}^j s_i=1, (t_n,\dots,t_j)\in[0,1]^{n-j+1}\}
\]

Let 
\[r_s(\gs)=\max_{(t_n,\dots,t_j)\in [0,1]^{n-j+1}}\{\diam H_{t_n}^{B(\sigma^n)}(H_{t_{n-1}}^{B(\sigma^{n-1})}(.. (H_{t_j}^{B(\sigma^j)}\{v(\gs^j)\})\}\},\]

where $v(\gs^j)$ denotes the set of  vertices of $\gs^j$.
Since all the homotopies $H_t^B$ take place in $X$, it follows that $|D_s(\gs)|\subset N_{r_s(\gs)}(X)$.

Therefore to prove~\eqref{i:mlsmall simplex}, it is enough to show that there is a bounded set $Z$ so that for any $\epsilon>0$ and $k\in \nn$, we can choose small neighborhood $W_x$ of $x\notin Z$ so that for any simplex $\gs\in W_x^{n+1}$, $n\leq k$, $s\in \mathcal{P}(\gs)$, we have $r_s(\gs)< \epsilon$.

Recall that we picked a locally finite  cover $\mathcal{U}$ of $X$ at the beginning of the proof. 
Hence there are only finitely many $U_\alpha\in \mathcal{U}$ containing any given $x\in X$. 
Let $V_x$ be the intersection of those finitely many $U_\alpha$.
Then for any 1-simplex $\gs\in V^2_x$, $B(\gs)$ is one of such $B(U_\alpha)$ by construction. 
In particular $\#\{B(\sigma)\mid \sigma\in V_x^2\}$ is finite.
Note that for any simplex $\sigma$, $B(\sigma)$ is determined by $\cup_{\tau\in S}B(\tau)$ and $\dim(\sigma)$, where $S$ is the set of 1-subsimplices of $\sigma$.
This implies that,
for any given $k\in \nn$, the set
$\{B(\sigma^n)\mid \sigma^n\in V_x^{n+1}, n\leq k\}$ is finite.

This implies that for any $\epsilon>0$, we can take small enough neighborhood $W_x \subset V_x$ of every $x\notin Z$,
such that $\diam\{H_{t_n}^{B(\sigma^n)} (H_{t_{n-1}}^{B(\sigma^{n-1})}..( H_{t_j}^{B(\sigma^j)}(W_x))..)\}<\epsilon$ for all $\sigma\in W_x^{n+1}$, for all $(\gs^n,\dots,\gs^j)\in \mathcal{P}(\gs)$, $(t_n,\ldots,t_j)\in [0,1]^{n-j+1}$ and $n\leq k$.
That implies $r_s(\gs)<\epsilon$ for all $s\in \mathcal{P}(\gs)$, $\sigma\in W_x^{n+1}$, $n\leq k$,  and $x\notin Z$.
That finishes the proof.

\end{proof}

In the next lemma we use the map $f$ and $D$ from   Lemma \ref{thin filling} to obtain a filling $\cff(X)\rightarrow \cs(N(X))$ that sends small enough simplices to their convex fillings, in particular to  small singular simplices.
\begin{Lemma}\label{thin filling 2}
Suppose $X\subset \ell^\infty$ is uniformly contractible at infinity. 
Let $N(X)$ be an open neighborhood of $X$ in $\ell^\infty$.
Then there exists two non-decreasing sequence of control functions $\rho_n,\mu_n:\rr_{\geq0}\rightarrow \rr_{\geq0}$ and a subcomplex $\cf^F(X)$ of $\cf(X)$ of the following form
\[\cf[n]^F
(X) = \langle \sigma^n \mid  d(\sigma^n; b) > \mu_n(\diam (\sigma^n))\rangle \subset  \cf[n](X),
\]
a chain map $g:\cf^F(X)\rightarrow \cs(N(X))$ with the following properties.
\begin{enumerate}
\item \label{|g|}$|g(\sigma^n)|\subset N_{\rho_n(\diam(\sigma^n))}(\gs^n)$ where the  tubular neighborhood is taken in $\ell^\infty$.
    \item \label{small for g} There exists a bounded  set $Z\subset X$  such that for each $x\in X-Z$, there is an open neighborhood $W_x$ of $x$ in $X$ such that  $g(\sigma)=c(\sigma)$ for all $\sigma\in W_x^{n+1}$ with $n\leq k$.
\end{enumerate}

\end{Lemma}
\begin{proof}

Let $f$, $D$, $Z$ and $W_x$ be as in the Lemma~\ref{thin filling}.
We define $g(\sigma):=c(\sigma)$  if $\sigma$ is  supported in $W_x$ for all $x\in X-Z$.
If a $i$-simplex $\sigma^i$ is not in $W_x^{i+1}$ for any $x$ but there is a $(i-1)$-subsimplex $\sigma'$ of $\sigma$ that lives in some $W^{i}_x$, then we define $g( \sigma):=f(\sigma)-[\sigma:\sigma']D(\sigma')$ (see figure \ref{fig:repairfilling}), where $[\gs:\gs']=\pm 1$ tells us whether orientations of $\gs$ and $\gs'$ agree or not.
For such $\sigma$, we will now check that $g(\partial \sigma)=\partial g(\sigma)$. For that we will use the fact that $D$ is a chain homotopy between $f$ and $c$. We first observe the following.
\begin{align*}
    \partial g(\sigma)&=\partial f(\sigma)-[\sigma:\sigma']\partial D(\sigma')\\
    &=f(\partial \sigma)+[\sigma:\sigma']D(\partial \sigma')-[\sigma:\sigma']f(\sigma')+[\sigma:\sigma']c(\sigma')\\
    \end{align*}
    Now we will show that the above is the same as $g\partial(\sigma)$. 
    First we observe that any $(i-1)$-subsimplex $\tau\neq \sigma'$  of $\sigma$ has a maximal $(i-2)$-subsimplex $\tau'$ supported in $V$.
    Hence for such $\tau$, we have $g(\tau)=f(\tau)-[\tau:\tau']D(\tau')$.
    We use this to obtain the following
    \begin{align*}
    g(\partial(\sigma))&=\sum_{\substack{\tau\in \partial \sigma\\\tau\neq \sigma'}}[\sigma:\tau]g(\tau)+[\sigma:\sigma']g(\sigma')\\
    &=\sum_{\substack{\tau\in \partial \sigma\\\tau\neq \sigma'}}[\sigma:\tau]g(\tau)+[\sigma:\sigma']c(\sigma')\\
    &=\sum_{\substack{\tau\in \partial \sigma\\\tau\neq \sigma'}}[\sigma:\tau]f(\tau)-\sum_{\substack{\tau\in \partial \sigma\\\tau\neq \sigma'}}[\sigma:\tau][\tau:\tau']D(\tau')+[\sigma:\sigma']c(\sigma')\\
    &=f(\partial \sigma)-[\sigma:\sigma']f(\sigma')-\sum_{\substack{\tau\in \partial \sigma\\\tau\neq \sigma'}}[\sigma:\tau][\tau:\tau']D(\tau')+[\sigma:\sigma']c(\sigma')
\end{align*}
To get $\partial g(\sigma)=g(\partial \sigma)$, it is now enough to show that 
\[[\sigma:\sigma']D(\partial \sigma')=-\sum_{\substack{\tau\in \partial \sigma\\\tau\neq \sigma'}}[\sigma:\tau][\tau:\tau']D(\tau')
\]
For that, we need to show that for a common $(i-2)$-subsimplex $\beta$ of $\sigma'$ and $\tau'$, $[\sigma:\sigma'][\sigma':\beta]=-[\sigma:\tau][\tau:\beta]$.
This is true because  coefficient of $\beta$ in $\partial (\partial \sigma)$ is $[\sigma:\sigma'][\sigma':\beta]+[\sigma:\tau][\tau:\beta]$  and so this has to be zero.

In other cases, an $i$-simplex $\sigma$ has some maximal subsimplices $\{\sigma'\}$ in $W_x$ for some $x$ with dimension less than $i-1$.
In this case, by induction on dimension, we can see that   $g(\partial \sigma)$ deforms retract to $f(\partial \sigma)=\partial (f(\sigma))$ by deforming along $D(\sigma')$ for each maximal subsimplex $\sigma'$ supported in $V$.
Then  we  use $f(\sigma)$ to fill this deformed $g(\partial \sigma)$. This process gives a cell whose boundary is $g(\partial \gs)$ and we define $g(\gs)$ to be that simplex.

Since $|f(\sigma^n)|\subset N_{\rho_n(\diam(\sigma^n))}(\sigma^n)$ and $|D(\sigma^n)| \subset N_{\rho_n(\diam(\sigma^n))}(\sigma^n)$ by Lemma~\ref{thin filling}, 
it follows that $|g(\sigma^n)|\subset N_{\rho_n(\diam(\sigma^n))}(\gs^n)$ giving us property~\eqref{|g|}.

 Finally, property~\eqref{small for g} follows directly from the construction of $g$.
\end{proof}
\begin{figure}
    \centering
    \includegraphics[scale=.8]{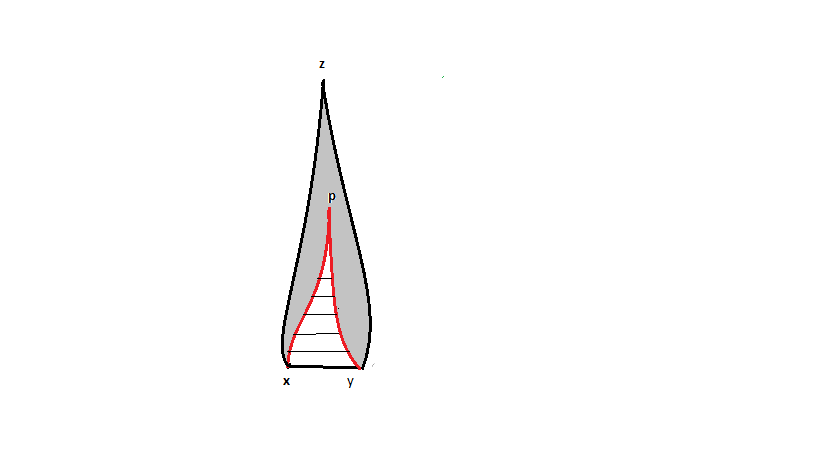}
    \caption{This is a picture of $g(\gs)$ where $\gs=(x,y,z)$.
    Here $p$ is the contracting point of the homotopy $H^{B((x,y))}$.
    The grey part is $f(\gs)$ and striped part is $D(\gs')$ where $\gs'=(x,y)$}
    \label{fig:repairfilling}
\end{figure}

We are now ready to prove~\ref{main lemma}. For the convenience of the reader, we recall the statement first.
\begin{Prop}\label{p:main prop}
 Suppose $X\subset \ell^\infty$ is  uniformly contractible  at infinity. 
 Let $N(X)$ be an open neighborhood of $X$ in $\ell^\infty$. Let $\mathcal{U}$ be an open cover of $N(X)$.
 Then there exist two non-decreasing sequence of control functions $\rho_n, \mu_n:\rr_{\geq 0}\rightarrow \rr_{\geq0}$ and a subcomplex $\cf^F(X)$ of $\cf(X)$ of the following form
\[\cf[n]^F
(X) = \langle \sigma^n \mid  d(\sigma^n; b) > \mu_n(\diam (\sigma^n))\rangle \subset  \cf[n](X),
\]
 a chain map $S:\cf^F(X)\rightarrow \cf^{\mathcal{U}}(N(X))$ and a map $G:\cf^F(X)\rightarrow \cf[*+1](N(X))$  with the following properties.
     \begin{enumerate}
    
    \item \label{i:G} $G$ is a chain homotopy between $VS$ and the inclusion map $i:\cf^F(X)\hookrightarrow\cf(N(X))$. 
    \[\d G(\gs)=VS(\gs)-i(\gs)-G\d(\gs)\]
    \item \label{|G|}
     $|G(\sigma^n)|\subset N_{\rho_n(\diam(\sigma^n))}(\sigma^n)$.
        \item\label{small}For any $k\in \nn$, there exists a bounded set $B\subset X$ such that for any $x\in X-B$ and a neighborhood $U$ of $x$ in $\ell^\infty$, there is a neighborhood $W$ of $x$ in $X$ such that $|G(\sigma^n)|\subset U^{n+2}$ for all $\sigma^n\in W^{n+1}$ if $n\leq k$.
    \end{enumerate}
\end{Prop}
\begin{proof}
Let us first define the maps $S$ and $G$.
We take the filling $g:\cf^F(X)\rightarrow \cs(N(X))$ from the Lemma~\ref{thin filling 2} and then apply barycentric subdivision on $g$ to define $S:\cf^F(X)\rightarrow \cf^\mathcal{U}(N(X))$. 
Since $|g(\gs^n)|\subset N_{\rho_n(\diam(\gs^n))}(\gs^n)$ by Lemma~\ref{thin filling 2}, it follows that $|S(\gs^n)|\subset N_{\rho_n(\diam(\gs))}(\gs^n)$ and consequently $|VS(\gs^n)|\subset N_{\rho_n(\diam(\gs))}(\gs^n)$.
Applying the Lemma~\ref{l:close maps are homotopic} to the map $VS$, we get the homotopy $G$ between $VS$ and the inclusion map $i:\cff(X)\rightarrow \cf(N(X))$ with property \eqref{|G|}.
Property~\eqref{small}   follows directly from the property~\eqref{small for g}  of the map $g$ from Lemma~\ref{thin filling 2}.
\end{proof}

Recall that the Proposition~\ref{G and S} was the main ingredient to prove the Theorem~\ref{warm up theorem}.
Likewise
Proposition~\ref{p:main prop} is `almost' enough to prove Theorem~\ref{main theorem}. 
The only caveat is that the maps $S$ and $G$ in Proposition~\ref{p:main prop} take images in $\cf(N(X))$ instead of $\cf(X)$ (compare with Proposition~\ref{G and S}).
Fortunately, thin enough neighborhoods of $X$ are as good as $X$ from the point of view of boundedly supported cohomology.
More precisely, we will show (in Proposition \ref{proj map and nbd}) that $\Hb(X)$ is taut in the sense that any class in $\Hb(X)$ is a restriction of a class from $\Hb(N(X))$ for some neighborhood $N(X)$ of $X$ inside $\ell^\infty$ .

Let $\mathcal{U}$ be a collection of open sets in $X$ and $A$ be a subset of $X$.
The \emph{star of $A$ with respect to $\mathcal{U}$}, denoted by $st(A,\mathcal{U})$, is defined to be the union of those elements of $\mathcal{U}$ whose intersection with $A$ is nonempty.
An open covering of $A$ in $X$ is a collection $\mathcal{U}$ of open sets of $X$ such that $A\subset st(A, \mathcal{U})$.
We now state the following lemma from \cite{AS} which roughly says that there is a projection map from $st(X,\mathcal{V})$ to $V$ that does not move close points too far apart.
\begin{Lemma}[\cite{AS}]\label{tautness criterion}
    If $X\subset \ell^\infty$, then for every open covering $\mathcal{U}$ of $X$ in $\ell^\infty$, there is an open covering $\mathcal{V}$ of $X$ in $\ell^\infty$ and a function $f:st(X,\mathcal{V})\rightarrow X$ such that
    \begin{enumerate}
        \item $f(x)=x$ for all $x\in X$
        \item For each $V\in \mathcal{V}$ with $V\cap X \neq \emptyset$ there is a $U\in \mathcal{U}$ such that $V\cup f(V)\subset U$.
    \end{enumerate}
    
\end{Lemma}
\begin{Definition}
A map $f:X\rightarrow Y$ between metric spaces is called \emph{coarse} map, if inverse image of a bounded set is bounded and there exist a non decreasing function $\rho:\rr_+\rightarrow \rr_+$ such that 
\[d(f(x),f(y))\leq \rho (d(x,y))
\]
\end{Definition}
The next proposition basically says that a boundedly supported cochain in $X$ can be extended to a boundedly supported cochain in $N(X)$ for an appropriate neighborhood $N(X)$ of $X$ in $\ell^\infty$.
\begin{Prop}\label{proj map and nbd}
    Let $X$ be a subspace of $\ell^\infty$. Then for any  $\phi\in \Cb(X)$, there exists a neighborhood $N(X)$ of $X$ in $\ell^\infty$ and a coarse map $f:N(X)\rightarrow X$ such that $f^*(\phi)\in \Cb(N(X))$ and $i^*f^*(\phi)=\phi$ where $i:X\rightarrow N(X)$ is the inclusion.
\end{Prop}
\begin{proof}
The proof is similar to the proof of tautness of Alexander-Spanier cohomology~\cite{AS}.
Suppose, $\phi\in \Cb(X)$.
That means there is a bounded set $B$ such that
for any $a\in X-B$, we can choose a neighborhood $U_a$ so that $U_a^{*+1}\cap |\phi|=\emptyset$.
Cover $B$ with a bounded open set $U_B$ in $X$.
Let $\mathcal{U}=U_B\cup \{U_a\}_{a\in X-B}$ be the open cover of $X$.
We can choose these open sets so that their diameters are less than $r$ for some $r\geq 0$.
Then Lemma \ref{tautness criterion} yields a refinement of  $\mathcal{V}$ of $\mathcal{U}$, and a map $f:st(X,\mathcal{V})\rightarrow X$ such that 
$f(a)=a$ for all $a\in X$ and for each $V\in \mathcal{V}$ with $V\cap X\neq \emptyset$, there is a $U\in \mathcal{U}$ such that $V\cup f(V)\subset U$.

By the second property of $f$ in Lemma~\ref{tautness criterion}, it follows that $d(x,f(x))\leq r$ for all $x\in st(X,\mathcal{V})$.
Hence, $d(f(x),f(y))\leq 2r+d(x,y)$ by triangle inequality which means $f$ has an upper control function. 
Furthermore $f$ sends unbounded sets to unbounded sets because $d(y,f(x))\geq d(y,x)-r$ by triangle inequality. Hence $f$ is a coarse map.

We let $N(X):=st(X,\mathcal{V})$. To show that $f^*(\phi)\in \Cb(N(X))$, we notice that if an element $V\in \mathcal{V}$ is far from $U_B$, then $V\cup f(V)$ does not touch $U_B$, and hence $f^*(\phi)|_{V}\subset U_a$ for some $a\in X-B$.
Finally $i^*f^*(\phi)=\phi$ because $f|_X=id_X$.
\end{proof}

Now we are ready to prove our main theorem which basically mimics the proof of Theorem~\ref{easier theorem}.

\begin{proof}[\textbf{Proof of \ref{main theorem}}]
	By the long exact sequence, we need to show $\H(\Cb(X)/\Cx(X))=0$.
 In other words, for $\phi \in \Cb[n](X)$ with $d\phi \in \Cx[n+1](X)$ we need to find $\psi \in \Cb[n-1](X)$ so that $\phi - d\psi \in \Cx[n]$.
	By Kuratowski's embedding theorem, $X$ can be embedded inside $\ell^\infty$.
	Proposition \ref{proj map and nbd} produces a neighborhood of $N(X)$ and a coarse map $f:N(X)\rightarrow X$ such that  $\phi'=f^*(\phi)\in \Cb(N(X))$.
	Let $U$ be a bounded neighborhood of $ \spp{\phi' }$ in $X$ and for each $x \in X - \spp{\phi' } $ choose a neighborhood $U_{x}$ such that $U_{x}^{n+1} \cap \Supp{\phi}=\emptyset$.
	Let $\cU$ denote the collection of $U_{x}$ together with $U$.
	Moreover, we have $d(\phi')=d(f^*(\phi))=f^*(d(\phi))$ is coarse  since $d(\phi)$ is coarse and $f$ is a coarse map.
	
	Proposition \ref{main lemma} produces the chain homotopy $G:\cf^F(X)\rightarrow \cf[*+1](N(X))$, which we use to define a linear map $D:\cf(X)\rightarrow \cf[*+1](N(X))$ by setting
	
	\[D(\sigma^n)=\begin{cases}G(\sigma^n)& \text{if $\sigma^n \in \cf[n]^F(X)$ }
	 \\
	0 & \text{otherwise}
	\end{cases}
	\]
	
Let $i:\cf(X)\rightarrow \cf(N(X))$ be the inclusion map.
	We now define
	\[
		T= i - \partial D - D \partial .
	\]
        And dually

	\[
		T^{*} = i^* -  D^{*} d - d D^{*}
	\]
By construction
 \[
 \phi-d D^*(\phi')=i^*(\phi')-dD^*(\phi')=T^*(\phi')+D^*d(\phi').
 \]
  Thus  $\psi=D^{*}(\phi')$ would be the desired cochain if we can prove $T^*(\phi')$ and $D^*d(\phi')$ are coarse cochains and $D^*(\phi')$ is a boundedly supported cochain.
  
  We claim that $T^{*}(\phi')$ is coarse.
   By Proposition~\ref{main lemma}\eqref{i:mlchain homotopy}, $T(\gs)=VS(\gs)$ for all $\gs\in \cff[n](X)$.
   Let $\gs^{n}$ be of some fixed diameter $r$.
If $d(\gs^{n},b) >  \mu_{n}(r)$, then $\gs^{n} \in \cff[n](X)$ and it follows that $T(\gs)$ is supported by $\cU$.
	 If all vertices of $\gs^{n}$ are outside of $\rho_{n}(r)$-neighborhood of $U$, then $\Supp{T(\gs)}$ does not meet $U^{n+1}$ and therefore $T^{*}(\phi')(\gs)=\phi'(T(\gs))=0$, since $\phi |_{U^{n+1}_{x}}=0$ for all $U_{x}$. 
	 Since $U$ is bounded, the claim follows.
	
	We claim that $D^{*}d(\phi')$ is coarse.
 Recall from Lemma~\ref{main lemma}\eqref{i:mlsupport control} that $\Supp{G\gs}\subset N_{\rho_{n}(\diam(\gs))}(\gs)$.
 Therefore, by construction $\Supp{D\gs}\subset N_{\rho_{n}(\diam(\gs))}(\gs)$.
	So, $D^{*}$ preserves coarseness, and the claim follows since $d(\phi')$ is coarse.

Finally, we claim that $D^*(\phi')\in \Cb[*+1](X)$.
Because of the property~\ref{main lemma}\eqref{i:mpsmall simplex} of $G$, we can choose a neighborhood $V$ of $x$ for all $x\notin ||\phi'||$ except points in some bounded set $B$, such that $D_*(V^{*+1})$ does not intersect $|\phi'|$.
That implies $x\notin ||D^*\phi'||$ for all $x \notin ||\phi'||\cup B$.
Since $\phi'\in \Cb[*](N(X))$, the claim follows.

\end{proof}

Combining Theorem \ref{main theorem} and Proposition \ref{p:bdd=singular} we get the following.
\begin{Corollary}\label{Hx=Has}
If $X$ is unbounded,  uniformly contractible at infinity, 
then for any $b\in X$
\[\Hx(X; R)=\begin{cases} 0 & \text{if  $*=0$}\\   \varinjlim \Hr[*-1](X-N_r(b); R) & \text{otherwise}
    \end{cases}\]

\end{Corollary}

\section{Computation of Coarse cohomology of the complement}\label{coarse and AS}
 We start by briefly reviewing the notion of coarse complement from~\cite{BB21}.
 Roughly the idea is that coarse complement of a subset $A$ in $X$ is determined by the collection of subsets $\mathcal{S}$ of $X$ which are \textit{coarsely disjoint} from $A$:
\[\mathcal{S}:=\{B\subset X \mid N_r(B) \cap N_r(A) \text{ is bounded for any $r\geq 0$} \}
\]
Coarse cohomology of the complement of $A$ is defined to be the cohomology of the following complex 
\[	\Cx[n](X-A)= \{ \phi \in \C[n](X) \mid \forall B \in \mathcal{S} \quad \phi|_{B}\in \Cx[n](B)\}
\]
with the usual coboundary operator $d$.
We denote the cohomology of the above complex by $\Hx(X-A)$.
It was shown in \cite{BB21} that $\Hx(X-A)$ can be computed as coarse cohomology of a single space which
we now recall from \cite{BB21}.
Consider the following pseudometric $d_A$ on $X$.
\[d_A(x,y):=\min\{d(x,A)+d(y,A), d(x,y)\}
\]
Note that $d_A(x,y)=0$ if and only if $(x,y)\in \bar{A} \times \bar{A}$ where $\bar{A}$ is the closure of $A$ in $X$.
Hence the pseudometric $d_A$ becomes a metric on the quotient space $X/\bar{A}$.
\begin{Prop}[\cite{BB21}]\label{coa comp space}
The quotient map $q:X\rightarrow X/\bar{A}$ induces isomorphism $\Hx(X-A)\cong\Hx(X/\bar{A})$.
\end{Prop}
 
Proposition~\ref{coa comp space} allows us to use Theorem~\ref{main theorem} to compute coarse cohomology of the complement.
 
 \begin{Theorem}\label{t:coarse complement}
 If $X/\bar{A}$ is uniformly contractible at infinity 
 and unbounded, then
 \begin{align*}
     \Hx(X-A)
    = &\begin{cases} 0 & \text{if  $*=0$}\\   \varinjlim \Hr[*-1](X-N_r(A)) & \text{otherwise}
    \end{cases}
\end{align*}
 \end{Theorem}
 
\begin{proof}
If  $X/\bar{A}$ is uniformly contractible at infinity, then Theorem~\ref{main theorem} yields $\Hx(X/\bar{A})=\Hb(X/\bar{A})$.
Combining this with Proposition~$\ref{coa comp space}$ gives us $\Hx(X-A)=\Hb(X/\bar{A})$.

Since $X/\bar{A}$ is unbounded, 
 for any $b\in X/\bar{A}$, 
Proposition~\ref{p:bdd=singular} gives us the following
\begin{align*}
    \Hb(X/\bar{A})
    = &\begin{cases} 0 & \text{if  $*=0$}\\   \varinjlim \Hr[*-1](X-N_r(A)) & \text{otherwise}
    \end{cases}
\end{align*}
Finally we observe that  $\varinjlim \Hr(X/\bar{A}-N_r(b))=\varinjlim \Hr(X-N_r(A))$. That finishes the proof.
\end{proof}
$X$ is called \emph{uniformly contractible away from $A$} if there are two non-decreasing function $\mu,\rho:[0,\infty)\rightarrow [0,\infty)$ such that any ball $B(x,r)$ of radius $r$ centered at $x\in X$ with $d(x,A)\geq \mu(r)$ is contractible inside $B(x,\rho(r))$. 
 
 Let $q:X\rightarrow X/\bar{A}$ be the quotient map.
 We observe that the  ball $B_r(x)$ in $(X,d)$ with center at $x$ is isometric to the ball $B_r(q(x))$ in $(X/\bar{A},d_A)$ if $d(x,A)>2r$.
 That means, if $X$ is uniformly contractible away from $A$,  
 then so is $X/\bar{A}$.
 If $X-N_r(A)\neq \emptyset $ for any $r$, then $X/\bar{A}$ id unbounded.
 Hence, as a consequence of Theorem \ref{t:coarse complement}, we get the following.

\begin{Corollary}
If $X$ is uniformly contractible away from $A$, 
and $X-N_r(A)\neq \emptyset $ for any $r$, then
 \begin{align*}
     \Hx(X-A)
    = &\begin{cases} 0 & \text{if  $*=0$}\\   \varinjlim \Hr[*-1](X-N_r(A)) & \text{otherwise}
    \end{cases}
\end{align*}
\end{Corollary}

\end{document}